\documentclass[11pt,a4paper]{article}
\usepackage{pifont}
\usepackage{CJK,epsfig,amsfonts,amsgen,amsmath,amstext,amsbsy,amsopn,amsthm
}
\usepackage{mathrsfs}
\usepackage{float}
\usepackage{epsfig,latexsym,amsfonts,amsmath,amssymb}
\usepackage{color}
\usepackage{calc}
\usepackage{tikz}

\usepackage{comment}
\usepackage{geometry}
\usepackage{calc}
\usepackage{tikz}
\usetikzlibrary{decorations.markings}
\tikzstyle{vertex}=[circle, draw, inner sep=0pt, minimum size=6pt]

\newtheorem{theorem}[subsection]{Theorem}
\newtheorem{lemma}[subsection]{Lemma}

\newtheorem{problem}[subsection]{Problem}

\usepackage{indentfirst} 
\newcounter{Hcase}

\theoremstyle{plain}

\date{\today}
\begin{document}

\title{Non-uniform Cross-intersecting Families}

\author{\small Zhen Jia${\thanks{E-mail address: 12231271@mail.sustech.edu.cn}}^{~1}$,~Qing Xiang${\thanks{Research partially supported by the National Natural Science Foundation of China Grant (No. 12071206, 12131011). E-mail address: xiangq@sustech.edu.cn}}^{~1}$,  Jimeng Xiao${\thanks{Research supported by Shenzhen Science and Technology Program (No. RCBS20221008093102011). E-mail address: xiaojm@sustech.edu.cn}}^{~1}$ and Huajun Zhang${\thanks{Research partially supported by the National Natural Science Foundation of China Grant (No. 12371332, 11971439). E-mail address: huajunzhang@usx.edu.cn}}^{~2}$
 \\[2mm]
\small ${}^1$  Department of Mathematics\\[-0.8ex]
\small Southern University of Science and Technology\\[-0.8ex]
\small Shenzhen, Guangdong 518055, China\\
\small ${}^2$Department of Mathematics\\[-0.8ex]
\small Shaoxing University \\[-0.8ex]
\small Shaoxing, Zhejiang 312000, China}


\baselineskip 20pt

\date{}
\maketitle
 \vspace{4mm}

\begin{abstract}
 Let $m\geq 2$, $n$ be positive integers, and $R_i=\{k_{i,1} >k_{i,2} >\cdots> k_{i,t_i}\}$ be subsets of $[n]$ for $i=1,2,\ldots,m$. The families $\mathcal{F}_1\subseteq \binom{[n]}{R_1},\mathcal{F}_2\subseteq \binom{[n]}{R_2},\ldots,\mathcal{F}_m\subseteq \binom{[n]}{R_m}$ are said to be non-empty cross-intersecting if for each $i\in [m]$, $\mathcal{F}_i\neq\emptyset$  and for any $A\in \mathcal{F}_i,B\in\mathcal{F}_j$, $1\leq i<j\leq m$, $|A\bigcap B|\geq1$. 
 In this paper, we determine the maximum value of $\sum_{j=1}^{m}|\mathcal{F}_j|$ for non-empty cross-intersecting family $\mathcal{F}_1, \mathcal{F}_2,\ldots,\mathcal{F}_m$ when $n\geq k_1+k_2$, where $k_1$ (respectively, $k_2$) is the largest (respectively, second largest) value in $\{k_{1,1},k_{2,1},\ldots,k_{m,1}\}$.
This result is a generalization of the results by Shi, Frankl and Qian \cite{shi2022non} on non-empty cross-intersecting families. Moreover, the extremal families are completely characterized.

\vskip 0.1in \noindent%
\textbf{Keywords}: EKR theorem, non-empty cross-intersecting families; the generating set method 
 \\[7pt]
{\sl MSC:}\ \ 05D05
\end{abstract}


\maketitle
\section{Introduction}

In this paper, we consider non-uniform cross-intersecting families. First we fix some notation. Throughout this paper, small letters are used to denote integers, capital letters are used for sets of integers, and script capital letters are used for collections of sets. The set of integers from $1$ to $n$ inclusive is denoted by $[n]$. A $k$-set (or $k$-subset) is a set of cardinality $k$ and the collection of all $k$-subsets of $[n]$ is denoted by $\binom{[n]}{k}$.

Let $n$ and $k$ be integers with $1\leq k\leq n$. A family $\mathcal{F}\subseteq \binom{[n]}{k}$ is said to be \emph{intersecting} if $|A\bigcap B|\geq1$ for any $A,B\in \mathcal{F}$. The celebrated Erd\H{o}s-Ko-Rado theorem determines the maximum size of intersecting families, and also characterizes the extremal families.
\begin{theorem} {\em \cite{erdos1961intersection}}
    Let $n\geq 2k$. If $\mathcal{F}\subseteq\binom{[n]}{k}$ is an intersecting family, then $|\mathcal{F}|\leq\binom{n-1}{k-1}$. Moreover, for $n>2k$, the equality holds if and only if $\mathcal{F}=\{A\in\binom{[n]}{k}:a\in A\}$ for some $a\in [n]$.
\end{theorem}

Cross-intersecting families are a variation of intersecting families. Let $m\geq2$ and $r_1,r_2,\ldots,r_m$ be positive integers. The families $\mathcal{F}_1\subseteq\binom{[n]}{r_1},\mathcal{F}_2\subseteq\binom{[n]}{r_2},\ldots,\mathcal{F}_m\subseteq\binom{[n]}{r_m}$ are said to be \emph{cross-intersecting} if $|A\bigcap B|\geq1$ for any $1\leq i<j\leq m$ and $A\in \mathcal{F}_i,B\in\mathcal{F}_j$. Cross-intersecting families $\mathcal{F}_1, \mathcal{F}_2,\ldots,\mathcal{F}_m$ are said to be non-empty if $\mathcal{F}_i\neq\emptyset$ for each $i\in [m]$. 

Hilton and Milner \cite{hilton1967some} determined the maximum sum of sizes of cross-intersecting families $\mathcal{F}_1,\mathcal{F}_2\subseteq\binom{[n]}{k}$; their result
is the first on the sum of sizes of cross-intersecting families.
\begin{theorem} {\em \cite{hilton1967some}}
     Let $n$ and $k$ be positive integers and $n\geq 2k$. Suppose that $\mathcal{F}_1,\mathcal{F}_2\subseteq\binom{[n]}{k}$ are non-empty cross-intersecting families. Then 
     \begin{equation*}
         |\mathcal{F}_1|+|\mathcal{F}_2|\leq 1+\binom{n}{k}-\binom{n-k}{k}.
     \end{equation*}
Equality holds if $\mathcal{F}_1=\{[k]\}$ and $\mathcal{F}_2=\{F\in\binom{[n]}{k}:|F\bigcap[k]|\geq1\}$.    
\end{theorem}

Borg and Feghali \cite{borg2022maximum} investigated non-uniform cross-intersecting families. 
Write $\binom{[n]}{\leq s}=\bigcup_{i=1}^{s}\binom{[n]}{i}$. These authors determined the maximum sum size and structure of extremal cross-intersecting families when $\mathcal{F}_1\subseteq\binom{[n]}{\leq r}$ and $\mathcal{F}_2\subseteq\binom{[n]}{\leq s}$. In \cite{FRANKL2024124},  Frankl, Liu, Wang and Yang used the generating set method to derive a weighted version of Borg and Feghali’s result. Let $R$ be a subset of $[n]$ and $\binom{[n]}{R}=\bigcup_{r\in R}\binom{[n]}{r}$. As a generalization, Li, Liu, Song and Yao \cite{li2023maximum} obtained an analogous result in the case where $\mathcal{F}_1\subseteq\binom{[n]}{R}, \mathcal{F}_2\subseteq\binom{[n]}{S}$ and $R,S$ are subsets of $[n]$.

\begin{theorem}{\em \cite{borg2022maximum}}
    If $n\geq 1$, $1\leq r\leq s$, $\mathcal{F}_1\subseteq\binom{[n]}{\leq r}$,  $\mathcal{F}_2\subseteq\binom{[n]}{\leq s}$, and $\mathcal{F}_1$ and $\mathcal{F}_2$ are non-empty and cross-intersecting, then
    \begin{equation*}
        |\mathcal{F}_1|+|\mathcal{F}_2|\leq 1+\sum_{i=1}^{s}\left(\binom{n}{i}-\binom{n-r}{i}\right)
    \end{equation*}
    and equality holds if $\mathcal{F}_1=\{[r]\}$ and $\mathcal{F}_2=\{B\in\binom{[n]}{\leq s}:B\bigcap [r]\neq\emptyset\}$.
\end{theorem}

Shi, Frankl, and Qian \cite{shi2022non} obtained a similar upper bound when an arbitrary number of cross-intersecting families are involved, and also characterized the extremal families $\mathcal{F}_1, \mathcal{F}_2,\ldots,\mathcal{F}_m$.

\begin{theorem}{\em \cite{shi2022non}\label{qianthm}}
    Let $m\geq2$ and $n,k$ be positive integers. Let $\mathcal{F}_1,\mathcal{F}_2,\ldots,\mathcal{F}_m\subseteq \binom{[n]}{k}$ be non-empty cross-intersecting families. If $n\geq 2k$, then 
    \begin{equation*}
    \sum_{j=1}^{m}|\mathcal{F}_j|\leq max \left\{ m\binom{n-1}{k-1}, \binom{n}{k}-\binom{n-k}{k}+m-1\right\}.   
    \end{equation*}
\end{theorem}

Shi, Frankl, and Qian \cite{shi2022non} also posed the following problem on non-uniform cross-intersecting families.

\begin{problem}{\em \cite{shi2022non}\label{Qianproblem}}
Let $m\geq 2$ and $n,k_1,k_2,\ldots,k_m$ be positive integers. Let $\mathcal{F}_1\subseteq \binom{[n]}{k_1},\mathcal{F}_2\subseteq \binom{[n]}{k_2},\ldots,\mathcal{F}_m\subseteq \binom{[n]}{k_m}$ be non-empty cross-intersecting families with $k_1\geq k_2\geq\cdots\geq k_m$ and $n\geq k_1+k_2$. Is it true that
\begin{equation*}
    \sum_{j=1}^{m}|\mathcal{F}_j|\leq max \left\{ \sum_{j=1}^{m}\binom{n-1}{k_j-1}, \binom{n}{k_1}-\binom{n-k_m}{k_1}+\sum_{i=2}^{m}\binom{n-k_m}{k_i-k_m}\right\}?
\end{equation*}
\end{problem}

Recently, Huang and Peng \cite{huang2023maximum}, as well as Zhang and Feng \cite{zhang2024note}, independently proved the above inequality using different methods. In this paper, we will give a common generalization of the above results.

Given positive integers $n,k$ with $k\leq n$ and a subset $R$ of $[n]$, define the families
\begin{equation*}
    \mathcal{M}_1(n,R,[k])=\{A\in\binom{[n]}{R}:|A\bigcap [k]|\geq 1\},
\end{equation*}
\begin{equation*}
    \mathcal{M}_2(n,R,[k])=\{B\in\binom{[n]}{R}:[k]\subseteq B\},
\end{equation*}
\begin{equation*}
   \mathcal{S}(n,R)=\{C\in\binom{[n]}{R}:1\in C\}.
\end{equation*}
It is easily checked that $\mathcal{M}_1(n,R,[k])$ and $\mathcal{M}_2(n,R,[k])$ are cross-intersecting. Below is our main result.

\begin{theorem}\label{maintheorem}
    Let $m\geq 2$, $n$ be positive integers, and $R_i=\{k_{i,1} >k_{i,2} >\cdots> k_{i,t_i}\}$ be subsets of $[n]$ for all $i=1,2,\ldots,m$. Let $\mathcal{F}_1\subseteq \binom{[n]}{R_1},\mathcal{F}_2\subseteq \binom{[n]}{R_2},\ldots,\mathcal{F}_m\subseteq \binom{[n]}{R_m}$ be non-empty cross-intersecting families such that $k_1$ (respectively, $k_2$) is the largest (respectively, second largest) value in $\{k_{1,1},k_{2,1},\ldots,k_{m,1}\}$ and $k_{min}^{\gamma}=min\{x\in R_j:j\in [m]\setminus\{\gamma\}\}$.
    If $n\geq k_1+k_2$, then 
$$
\sum_{j=1}^{m}|\mathcal{F}_j|\leq max\left\{\sum_{j=1}^{m}|\mathcal{S}(n,R_j)|
,|\mathcal{M}_{1}(n,R_{\gamma},[k_{min}^{\gamma}])|+\sum_{\alpha\in[m]\setminus\{\gamma\}}|\mathcal{M}_{2}(n,R_{\alpha},[k_{min}^{\gamma}])|:\gamma\in[m]
\right\}.
$$
Equality holds if and only if, up to permutations of the elements in $[n]$, one of the following holds:
\begin{itemize}
    \item[(i)] if $\sum_{j=1}^{m}|\mathcal{S}(n,R_j)|\geq  |\mathcal{M}_{1}(n,R_{\gamma},[k_{min}^{\gamma}])|+\sum_{\alpha\in[m]\setminus\{\gamma\}}|\mathcal{M}_{2}(n,R_{\alpha},[k_{min}^{\gamma}])|$ for all $\gamma\in[m]$, then $\mathcal{F}_j=\mathcal{S}(n,R_j)$ for all $j\in [m]$; 
    \item[(ii)] if $\sum_{j=1}^{m}|\mathcal{S}(n,R_j)|\leq  |\mathcal{M}_{1}(n,R_{\gamma},[k_{min}^{\gamma}])|+\sum_{\alpha\in[m]\setminus\{\gamma\}}|\mathcal{M}_{2}(n,R_{\alpha},[k_{min}^{\gamma}])|$ for some $\gamma\in[m]$, then $\mathcal{F}_{\gamma}=\mathcal{M}_{1}(n,R_{\gamma},[k_{min}^{\gamma}])$ and $\mathcal{F}_{\alpha}=\mathcal{M}_{2}(n,R_{\alpha},[k_{min}^{\gamma}])$ for all $\alpha\in[m]\setminus\{\gamma\}$;
    \item[(iii)] if $n=k_1+k_2$, $m=2$, $R_1=\{k_1\}$ and $R_2=\{k_2\}$, then $\mathcal{F}_1=\binom{[n]}{k_1}\setminus{\overline{\mathcal{F}_2}}$ and $\mathcal{F}_2\subseteq\binom{[n]}{k_2}$ with $0<|\mathcal{F}_2|<\binom{n}{k_2}$; (here $\overline{\mathcal{F}_2}$ is the complement of ${\mathcal F}_2$.)
    \item[(iv)] if $n=k_1+k_2$, $m\geq 3$ and $R_i=\{k\}$ for all $i\in[m]$, then $\mathcal{F}_j=\mathcal{F}$ for all $j\in[m]$ where $\mathcal{F}$ is an
    intersecting family with $|\mathcal{F}|=\binom{n-1}{k-1}$.
\end{itemize}
\end{theorem}

Note that in the special case where $R_j = \{k\}$ for all $j \in [m]$, Theorem \ref{maintheorem} reduces to Theorem \ref{qianthm}; and when $R_j = \{k_j\}$ for all $j \in [m]$, Theorem \ref{maintheorem} offers a positive answer to Problem \ref{Qianproblem}.

The methods employed in this paper are the shifting method and the generating set method. The definitions of the shifting operation and generating sets, along with their basic properties, are introduced in the next section. In Section 3, we give the detailed proof of Theorem \ref{maintheorem}. Finally, in Section 4, we discuss directions for future work. 

\section{Preliminaries}
Let $\mathcal{F}$ be a family in $\binom{[n]}{k}$. For $i,j\in[n]$ and $A\in \mathcal{F}$, define
$$
s_{i,j}(A)=\begin{cases}\begin{array}{cc}
    (A\setminus\{j\})\bigcup\{i\},&\text{if } j\in A,i\notin A, (A\setminus\{j\})\bigcup\{i\}\notin \mathcal{F};\\
    A,&\text{otherwise}
    \end{array}\end{cases}
$$
and $s_{i,j}(\mathcal{F})=\{s_{i,j}(A):A\in\mathcal{F}\}$. It is well known that
$s_{i,j}$ preserves the size $|\mathcal{F}|$, the intersecting property, and  the cross-intersecting property. We apply the shifting operations to $\mathcal{F}$ for all $1\leq i<j\leq n$ until we obtain a family $\mathcal{F}'$ such that $s_{i,j}(\mathcal{F}')=\mathcal{F}'$ for all $1\leq i<j\leq n$; such an $\mathcal{F}'$ is called \emph{left-compressed}. It is clear that $s_{i,j}(\mathcal{F}_1)$ and $s_{i,j}(\mathcal{F}_2)$ are cross-intersecting if and only if $\mathcal{F}_1$ and $\mathcal{F}_2$ are cross-intersecting. For more details, see the survey by Frankl \cite{frankl1987shifting}.

For a subset  $R\subseteq[n]$, a family $\mathcal{F}\subseteq\binom{[n]}{R}$ is called \emph{monotone} if $A\in\mathcal{F},A\subseteq B$ and $|B|\in R$ imply $B\in \mathcal{F}$. Denote $(\mathcal{F})_r=\{F\in\mathcal{F}:|F|=r\}$. Given a family $\mathcal{A}\subseteq\binom{[n]}{R}$, the \emph{up-set} of $\mathcal{A}$, denoted by $\langle{\mathcal{A}}\rangle_{R}$, is defined as follows:
\begin{equation*}
   \langle{\mathcal{A}}\rangle_{R}=\{F\in\binom{[n]}{R}:\text{there exists $A\in\mathcal{A}$ such that $A\subseteq F$}\}. 
\end{equation*}

Our proof is based on the generating set method, as outlined in \cite{ahlswede1997complete,FRANKL2024124,zhang2025cross-t-intersecting}. We will recall some well-known concepts related to the generating set method.

Let $\mathcal{F}\subseteq\binom{[n]}{R}$ be a monotone family with max$(R)=r$. A \emph{generating set} of $\mathcal{F}$ is a minimal (with respect to set containment) set $E\in\binom{[n]}{\leq r}$ such that $\langle{E}\rangle_R\subseteq \mathcal{F}$.  The \emph{generating family} of $\mathcal{F}$ consists of all generating sets of $\mathcal{F}$. The extent of $\mathcal{F}$ is defined as the maximal element appearing in a generating set of $\mathcal{F}$. For subsets $E,R\subseteq[n]$, denote $\mathscr{D}_R(E)=\{A\in\binom{[n]}{R}:A\bigcap [\text{max}(E)]=E\}$.  

The following properties of the generalized generating families introduced above are analogous to those of the generating sets discussed in \cite{ahlswede1997complete} and can be derived in a similar manner. 

\begin{lemma}\label{two-lemma}
Let $R_1,R_2$ be subsets of $[n]$, $\text{max}(R_1)=k_1$ and $\text{max}(R_2)=k_2$. Suppose that $\mathcal{F}_1\subseteq\binom{[n]}{R_1}(\text{respectively, }\mathcal{F}_2\subseteq\binom{[n]}{R_2})$ is a monotone and left-compressed family with the generating family $\mathcal{G}_1(\text{respectively, }\mathcal{G}_2)$ and extent $l_1$.
If $\mathcal{F}_1\subseteq\binom{[n]}{R_1}$ and $\mathcal{F}_2\subseteq\binom{[n]}{R_2}$ are non-empty and cross-intersecting and $n\geq k_1+k_2$, then
\begin{itemize}
    \item[(i)] $|E_1\bigcap E_2|\geq1$ for all $E_1\in \mathcal{G}_1$ and $E_2\in \mathcal{G}_2$;
    \item[(ii)] for any $1\leq x<y\leq l_1$ and $E\in \mathcal{G}_1$, we have $E_0\subseteq s_{x,y}(E)$ for some $E_0\in \mathcal{G}_1$.
\end{itemize}
\end{lemma}

\begin{proof}
    \begin{itemize}
        \item[(i)] If not, since $n\geq k_1 + k_2$, there exist $F_1 \in (\mathcal{F}_1)_{k_1}$ and $F_2 \in (\mathcal{F}_2)_{k_2}$ such that $E_1 \subseteq F_1,E_2 \subseteq F_2$, and $F_1 \bigcap F_2 = \emptyset$, contradicting that $\mathcal{F}_1$ and $\mathcal{F}_2$ are cross-intersecting.
        
        \item[(ii)] By the definition of generating sets, there exists an
        $F \in \mathcal{F}_1$ such that $E\subseteq F$. Applying the shifting operation $s_{x,y}$, we have $s_{x,y}(E) \subseteq s_{x,y}(F) \in \mathcal{F}_1$ since $s_{x,y}$ preserves $\mathcal{F}_1$. It follows that $\langle{s_{x,y}(E)}\rangle_{R}\subseteq \mathcal{F}_1$. Thus, there exists a generating set $E_0$ in $\mathcal{G}_1$ such that $E_0 \subseteq s_{x,y}(E)$.
    \end{itemize}
\end{proof}

\begin{lemma}\label{basiclemma}
Let $\mathcal{F}\subseteq\binom{[n]}{R}$ be a monotone family  
with the generating family $\mathcal{G}$.
Then $\mathcal{F}=\bigcup_{E\in \mathcal{G}}\langle{E}\rangle_{R}$, denoted by $\langle{\mathcal{G}}\rangle_{R}$. Furthermore, if $\mathcal{F}$ is left-compressed, then it is a disjoint union 
\begin{equation*}
    \mathcal{F}=\dot{\bigcup_{E\in\mathcal{G}}} \mathscr{D}_R(E).
\end{equation*}
\end{lemma}

\begin{proof}
    It is clear that $\mathcal{F}=\langle{\mathcal{G}}\rangle_{R}$ and ${\bigcup_{E\in\mathcal{G}}} \mathscr{D}_R(E)\subseteq\mathcal{F}$. Since each element in $\mathcal{G}$ is minimal, we see that ${\bigcup_{E\in\mathcal{G}}} \mathscr{D}_R(E)$ is a disjoint union. For any $G\in \mathcal{G}$ and $a\in R$, it suffices to show that $\langle{G}\rangle_{a}\subseteq {\bigcup_{E\in\mathcal{G}}} \mathscr{D}_a(E)$. We prove this by induction on the value max$(G)$. The base case is trivial as $\langle{G}\rangle_{a}=\mathscr{D}_a(G)$ if max$(G)=1$. Assume that  $\langle{G}\rangle_{a}\subseteq {\bigcup_{E\in\mathcal{G}}} \mathscr{D}_a(E)$ for all $G$ with max$(G)\leq \nu$. If max$(G)=\nu+1$, $\langle{G}\rangle_{a}=\{F\in\binom{[n]}{a}:G\subseteq F\}=\{F\in\binom{[n]}{a}:G= F\bigcap [\nu+1]\}\bigcup\{F\in\binom{[n]}{a}:G\subsetneq F\bigcap [\nu+1]\}$. The first part is contained in $\mathscr{D}_a(G)$. Notice that all elements in the second part are contained in $\langle{\{x\}\bigcup G}\rangle_a$ for some $x\in (F\setminus{G})\bigcap{[\nu+1]}$, which is a subfamily of $\langle{s_{x,\nu+1}(G)}\rangle_a$. By Lemma \ref{two-lemma}(ii), there exists a generating set $G_0\in \mathcal{G}$ such that $G_0\subseteq s_{x,\nu+1}(G)$ with $\langle{s_{x,\nu+1}(G)}\rangle_a\subseteq\langle{G_0}\rangle_a$. Since max$(G_0)\leq \nu$, we have $\{F\in\binom{[n]}{a}:G\subsetneq F\bigcap [\nu+1]\}\subseteq {\bigcup_{E\in\mathcal{G}}} \mathscr{D}_a(E)$ by the induction hypothesis. The proof is now complete.  
\end{proof}

Let $\prec_L$ be the lexicographic order on $\binom{[n]}{i}$
where $i\in [n]$, that is, for any
two sets $A, B\in \binom{[n]}{i}$
, $A \prec_L B$ if and only if $\text{min}\{a : a\in A\setminus B\}< \text{min}\{b : b \in B \setminus A\}$.  For a family $\mathcal{F}\subseteq\binom{[n]}{k}$
, let $\mathcal{F}_L$ denote the family consisting of the first $|\mathcal{F}|$ $k$-sets in order $\prec_L$ in $\binom{[n]}{k}$, and $\mathcal{F}$ is \emph{L-initial} if $\mathcal{F}_L=\mathcal{F}$. 
\begin{theorem}\label{keyleftlemma}\cite{hilton1976erdos,ou2005maximum}
    Let $\mathcal{A}\subseteq\binom{[n]}{k_1}$ and $\mathcal{B}\subseteq\binom{[n]}{k_2}$ be cross-intersecting families. Then  $\mathcal{A}_L$ and $\mathcal{B}_L$ are cross-intersecting as well.
\end{theorem}

\section{Proof of Theorem \ref{maintheorem}}

We are now in a position to prove the main result. 
\begin{proof}[\bf{Proof of Theorem \ref{maintheorem}}]
    Let $\mathcal{F}_1\subseteq\binom{[n]}{R_1},\mathcal{F}_2\subseteq\binom{[n]}{R_2},\ldots,\mathcal{F}_m\subseteq\binom{[n]}{R_m}$ be non-empty cross-intersecting families with maximum $\sum_{j=1}^{m}|\mathcal{F}_j|$. By Theorem \ref{keyleftlemma}, without loss of generality, we can assume that $(\mathcal{F}_j)_r$ is L-initial for all $j\in [m]$ and $r\in R_j$.
    By the maximality of $\sum_{j=1}^{m}|\mathcal{F}_j|$, we may also assume that all $\mathcal{F}_j$ where $j\in[m]$ are monotone.   
    Since the family $\mathcal{F}_j$ remains unchanged under the shifting operations, it follows that all $\mathcal{F}_j$ for $j\in[m]$ are left-compressed. 

    Suppose that $\mathcal{F}_j$ has extent $l_j$ and generating family $\mathcal{G}_j$ for all $j\in[m]$. Let $l=\text{max}\{l_j:j\in [m]\}$ and the boundary generating family $\overline{\mathcal{G}_{j}}$ of $\mathcal{F}_j$ consisting of all generating sets of $\mathcal{F}_j$ containing $l$ for all $j\in[m]$. 

    If $|\mathcal{F}_j|<\sum_{a\in R_j}\binom{n-1}{a-1}$ for all $j\in [m]$, then we are done as $|\mathcal{S}(n,R_j)|=\sum_{a\in R_j}\binom{n-1}{a-1}$. 
    
    Assume that $|\mathcal{F}_\gamma|\geq\sum_{a\in R_\gamma}\binom{n-1}{a-1}$ for some $\gamma\in [m]$. Since $(\mathcal{F}_\gamma)_{r}$ is L-initial and has a size greater than or equal to $\binom{n-1}{r-1}$ for some $r\in R_{\gamma}$, we infer that $\langle{1}\rangle_{r}\subseteq (\mathcal{F}_\gamma)_{r}$. Using $n> k_1+k_2$, we know that each element in $\mathcal{F}_{\alpha}$ for all $\alpha\in[m]\setminus\{\gamma\}$ contains $1$. By the maximality of $\sum_{j=1}^{m}|\mathcal{F}_j|$, we deduce that $\{1\}\in \mathcal{G}_\gamma$. If $l=1$, then $\mathcal{F}_j=\mathcal{S}(n,R_j)$ for all $j\in[m]$, and we are done. Therefore, we only need to consider the case when $l>1$.

    In fact, if $l>1$, then $\overline{\mathcal{G}_{\gamma}}\neq\emptyset$. To demonstrate this, we need the following Lemma \ref{keylemma}, which is analogous to a result in \cite{zhang2025cross-t-intersecting}. For the convenience of the reader, we provide a proof.  
\begin{lemma}\label{keylemma}
    For $1< u\leq l$, if $(\overline{\mathcal{G}_i})_u\neq\emptyset$ for some $i\in[m]$ and $\sum_{i=1}^{m}|\mathcal{F}_i|$ is maximum, then there exists $(\overline{\mathcal{G}_j})_{l+1-u}\neq\emptyset$ for some $j\in [m]\setminus\{i\}$. In other words, for any $E\in \mathcal{G}_i, F\in \mathcal{G}_j$, if $E\bigcap F=\{l\}$, then $E\bigcup F=[l]$ and $|E|+|F|=l+1$. Furthermore, $|(\overline{\mathcal{G}_i})_u|=|(\overline{\mathcal{G}_j})_{l+1-u}|$.
\end{lemma}
\begin{proof}[Proof of Lemma \ref{keylemma}]
   Notice that $(\overline{\mathcal{G}_i})_u\neq\emptyset$ means that there exists a generating set $E\in {\mathcal{G}_i}$ such that $|E|=u$ and $l\in E$. Consider $E'=E\setminus\{l\}\neq\emptyset$ as $u>1$. If for any $F_0\in {\mathcal{G}_j}$ where $j\in [m]\setminus\{i\}$ such that $F_0\bigcap E'\neq\emptyset$, then 
   \begin{equation*}
   (\mathcal{F}_1,\mathcal{F}_2,\ldots,\langle{\mathcal{G}_i\setminus\{E\}\bigcup \{E'\}}\rangle_{R_i},\ldots,\mathcal{F}_m)    
   \end{equation*}
   is a new cross-intersecting sequence, contradicting the maximality of $\sum_{i=1}^{m}|\mathcal{F}_i|$. Therefore, there exists an $F \in {\mathcal{G}_j}$ for some $j\in [m]\setminus\{i\}$ such that $E'\bigcap F=\emptyset$. By Lemma \ref{two-lemma}(i), we have $|E\bigcap F|=1$, which implies $l\in F$(i.e., $F\in \overline{{\mathcal{G}_j}}$) and $E\bigcap F=\{l\}$. 

   To complete the proof, we need to show that $E\bigcup F=[l]$. Clearly, $E\bigcup F\subseteq[l]$. On the contrary, if there exists an element $x\in[l]$ such that $x\notin E\bigcup F$, then by Lemma \ref{two-lemma}(ii), there exists an $F_1\in \mathcal{G}_j$ such that $F_1\subseteq s_{x,l}(F)$. This implies that $l\notin F_1\bigcap E\subseteq s_{x,l}(F)\bigcap E$. Since $x\notin E$ and $E\bigcap F=\{l\}$, we have $s_{x,l}(F)\bigcap E=\emptyset$. Thus $F_1\bigcap E=\emptyset$, contradicting Lemma \ref{two-lemma}(i). Therefore, $E\bigcup F=[l]$ and $F\in (\overline{\mathcal{G}_j})_{l+1-u}$. 
\end{proof}   

We claim that $(\overline{\mathcal{G}_{\gamma}})_u\neq\emptyset$ for some $1\leq u\leq l$. On the contrary, assume $\overline{\mathcal{G}_{\gamma}}=\emptyset$ and $\overline{\mathcal{G}_{\alpha}}\neq\emptyset$ for some $\alpha\in[m]\setminus\{\gamma\}$. Combining $\{1\}\in \mathcal{G}_{\gamma}$ with Lemma \ref{two-lemma}(i), we can select a generating set $E$ in $\overline{\mathcal{G}_{\alpha}}$ such that $1,l\in E$. By Lemma \ref{keylemma}, there exists an $F\in \overline{(\mathcal{G}_\beta)}$ for some  $\beta\in[m]\setminus\{\alpha,\gamma\}$ such that $E\bigcup F=[l]$ and $E\bigcap F=\{l\}$, which contradicts the fact that $1\in E\bigcap F$. Therefore, we conclude that $(\overline{\mathcal{G}_{\gamma}})_u\neq\emptyset$ for some $1\leq u\leq l$, and this can be divided into two cases according to the value of $u$. 

\textbf{Case 1.} $1<u<l$. We will construct two new cross-intersecting families and obtain a contradiction. Denote $(\overline{\mathcal{G}_{\gamma}})_{u}'=\{E\setminus\{l\}:E\in (\overline{\mathcal{G}_{\gamma}})_u\}$. By Lemma \ref{keylemma}, there exists $(\overline{\mathcal{G}_{\alpha}})_{l+1-u}\neq\emptyset$ for some $\alpha\in[m]\setminus\{\gamma\}$ and so $\mathscr{D}_{R_{\alpha}}((\overline{\mathcal{G}_{\alpha}})_{l+1-u})\neq\emptyset$. 
Let 
\begin{equation*}
\Tilde{\mathcal{F}_{\gamma}}=\mathcal{F}_{\gamma}\bigcup \mathscr{D}_{R_{\gamma}}((\overline{\mathcal{G}_{\gamma}})_{u}')\text{ and } 
\Tilde{\mathcal{F}_{\alpha}}=\mathcal{F}_{\alpha}\setminus \mathscr{D}_{R_{\alpha}}((\overline{\mathcal{G}_{\alpha}})_{l+1-u}) \text{ for all }\alpha\in[m]\setminus\{\gamma\}.
\end{equation*}

In the special case where $b<l+1-u$ for some $\alpha\in[m]\setminus\{\gamma\}$ and $b\in R_{\alpha}$, we adopt the convention that $\mathscr{D}_{b}((\overline{\mathcal{G}_{\alpha}})_{l+1-u})=\emptyset$. 

\begin{lemma}\label{new-cross-inter1}
$(\Tilde{\mathcal{F}_1},\Tilde{\mathcal{F}_2},\ldots,\Tilde{\mathcal{F}_{\gamma}},\ldots,\Tilde{\mathcal{F}_{m}})$ is a cross-intersecting sequence. Furthermore,
\begin{equation*}
\sum_{j=1}^{m}|\Tilde{\mathcal{F}_j}|=\sum_{j=1}^{m}|\mathcal{F}_j|+\sum_{a\in R_{\gamma}}|(\overline{\mathcal{G}_\gamma})_u|\binom{n-l}{a-u+1}-\sum_{\alpha\in[m]\setminus\{\gamma\}}
\sum_{b\in R_{\alpha}}|(\overline{\mathcal{G}_{\alpha}})_{l+1-u}|\binom{n-l}{b-(l+1-u)}.
\end{equation*}
\end{lemma}

\begin{proof}[Proof of Lemma \ref{new-cross-inter1}]
We first determine the cardinality of $\sum_{j=1}^{m}|\Tilde{\mathcal{F}_{j}}|$. By Lemma \ref{basiclemma}, we have $\Tilde{\mathcal{F}_{\gamma}}=\left({\dot{\bigcup}_{E\in{\mathcal{G}_{\gamma}}}}\mathscr{D}_{R_{\gamma}}(E)\right)\bigcup\mathscr{D}_{R_{\gamma}}((\overline{\mathcal{G}_{\gamma}})_{u}')$. Notice that 

\begin{align*}
 \mathscr{D}_{R_{\gamma}}((\overline{\mathcal{G}_{\gamma}})_{u}')&=  \bigcup_{\substack{a\in R_{\gamma},\\E\in (\overline{\mathcal{G}_{\gamma}})_{u}}}(E\setminus\{l\})\times\binom{\{l,l+1,\cdots,n\}}{a-u+1}   \\
 &=\bigcup_{\substack{a\in R_{\gamma},\\E\in (\overline{\mathcal{G}_{\gamma}})_{u}}}
 \left(E\times\binom{\{l+1,l+2,\cdots,n\}}{a-u}\bigcup(E\setminus\{l\})\times\binom{\{l+1,l+2,\cdots,n\}}{a-u+1}\right)\\
 &=\mathscr{D}_{R_{\gamma}}((\overline{\mathcal{G}_{\gamma}})_{u})\bigcup\left( \bigcup_{\substack{a\in R_{\gamma},\\E\in (\overline{\mathcal{G}_{\gamma}})_{u}}}(E\setminus\{l\})\times\binom{\{l+1,l+2,\cdots,n\}}{a-u+1}\right).
\end{align*}

For fixed $a\in R_{\gamma}$ and $E\in(\overline{\mathcal{G}_{\gamma}})_{u}$, the set $(E\setminus\{l\})\times\binom{\{l+1,l+2,\cdots,n\}}{a-u+1}$ cannot be in  $\mathscr{D}_{a}(E_0)$ for any $E_0\in \mathcal{G}_{\gamma}$, as 
$\mathcal{G}_{\gamma}$ contains minimal elements and $E_0\subsetneq E$. Thus $\Tilde{\mathcal{F}_{\gamma}}$ is a disjoint-union of $\mathcal{F}_{\gamma}$ and $\bigcup_{{a\in R_{\gamma},E\in (\overline{\mathcal{G}_{\gamma}})_{u}}}(E\setminus\{l\})\times\binom{\{l+1,l+2,\cdots,n\}}{a-u+1}$. It follows that   $|\Tilde{\mathcal{F}_{\gamma}}|=|\mathcal{F}_{\gamma}|+\sum_{a\in R_{\gamma}}|(\overline{\mathcal{G}_\gamma})_u|\binom{n-l}{a-u+1}$. Clearly, $\Tilde{\mathcal{F}_{\alpha}}=\left({\dot{\bigcup}_{E\in{\mathcal{G}_{\alpha}}}}\mathscr{D}_{R_{\alpha}}(E)\right)\setminus \mathscr{D}_{R_{\alpha}}((\overline{\mathcal{G}_{\alpha}})_{l+1-u})$ has the size $|\Tilde{\mathcal{F}_{\alpha}}|=|{\mathcal{F}_{\alpha}}|-
\sum_{b\in R_{\alpha}}|(\overline{\mathcal{G}_{\alpha}})_{l+1-u}|\binom{n-l}{b-(l+1-u)}$ for all $\alpha\in [m]\setminus\{\gamma\}$. 

Since $\Tilde{\mathcal{F}_{\alpha}},\Tilde{\mathcal{F}_{\beta}}$ for distinct $\alpha,\beta\in[m]\setminus\{\gamma\}$ are cross-intersecting, it suffices to show that $ \bigcup_{\substack{a\in R_{\gamma},E\in (\overline{\mathcal{G}_{\gamma}})_{u}}}(E\setminus\{l\})\times\binom{\{l+1,l+2,\cdots,n\}}{a-u+1}$ and $\mathcal{F}_{\alpha}\setminus \mathscr{D}_{R_{\alpha}}((\overline{\mathcal{G}_{\alpha}})_{l+1-u})=\bigcup_{F\in \mathcal{G}_{\alpha}\setminus{(\overline{\mathcal{G}_{\alpha}})_{l+1-u}}}\mathscr{D}_{R_{\alpha}}(F)$ 
are cross-intersecting. Choosing $E$ and $F$ as above implies $|E\bigcap F|\geq 1$ by Lemma \ref{two-lemma}(i). If $E\bigcap F\neq\{l\}$, then we are done. If $E\bigcap F=\{l\}$, $E\in (\overline{\mathcal{G}_{\gamma}})_{u}$ implies $F\in(\overline{\mathcal{G}_{\alpha}})_{l+1-u}$ by Lemma \ref{keylemma}, a contradiction. Therefore, the proof is completed. 
\end{proof}

Similarly, let 
\begin{equation*}
\hat{\mathcal{F}_{\gamma}}=\mathcal{F}_{\gamma}\setminus \mathscr{D}_{R_{\gamma}}((\overline{\mathcal{G}_{\gamma}})_{u})
\text{ and }
\hat{\mathcal{F}_{\alpha}}=\mathcal{F}_{\alpha}\bigcup\mathscr{D}_{R_{\alpha}}((\overline{\mathcal{G}_{\alpha}})_{l+1-u}') \text{ for all }\alpha\in[m]\setminus\{\gamma\}. 
\end{equation*}

\begin{lemma}\label{new-cross-inter2}
$(\hat{\mathcal{F}_1},\hat{\mathcal{F}_2},\ldots,\hat{\mathcal{F}_{\gamma}},\ldots,\hat{\mathcal{F}_{m}})$ is a cross-intersecting sequence. Furthermore,
\begin{equation*}
\sum_{j=1}^{m}|\hat{\mathcal{F}_j}|=\sum_{j=1}^{m}|\mathcal{F}_j|+\sum_{\alpha\in[m]\setminus \{\gamma\}}\sum_{b\in R_{\alpha}}|(\overline{\mathcal{G}_{\alpha}})_{l+1-u}|\binom{n-l}{b-(l+1-u)+1}-\sum_{a\in R_{\gamma}}|(\overline{\mathcal{G}_\gamma})_u|\binom{n-l}{a-u}.    
\end{equation*}
\end{lemma}

\begin{proof}[Proof of Lemma \ref{new-cross-inter2}]
   Using a similar argument as in the proof of Lemma \ref{new-cross-inter1}, we can calculate the value of $\sum_{j=1}^{m}|\hat{\mathcal{F}_j}|$ and show that $\hat{\mathcal{F}_{\gamma}}$ and $\hat{\mathcal{F}_{\alpha}}$ for each $\alpha\in[m]\setminus\{\gamma\}$ are cross-intersecting. Since $l+1-u\geq 2$ and $\{1\}\in \mathcal{G}_{\gamma}$, each element in $\overline{\mathcal{G}_{\alpha}}$ for any $\alpha\in[m]\setminus\{\gamma\}$ contains $1$. Thus, $\hat{\mathcal{F}_{\alpha}}$ and  $\hat{\mathcal{F}_{\beta}}$ are cross-intersecting for distinct $\alpha,\beta\in[m]\setminus\{\gamma\}$. 
\end{proof}

Since $\sum_{j=1}^{m}|\mathcal{F}_j|$ is maximum, it follows from $\sum_{j=1}^{m}|\mathcal{F}_j|\geq \sum_{j=1}^{m}|\Tilde{\mathcal{F}_j}|$ that 
    \begin{equation}{\label{equ-1}}
    \sum_{\alpha\in[m]\setminus \{\gamma\}}\sum_{b\in R_{\alpha}}|(\overline{\mathcal{G}_{\alpha}})_{l+1-u}|\binom{n-l}{b-(l+1-u)}\geq \sum_{a\in R_{\gamma}}|(\overline{\mathcal{G}_\gamma})_u|\binom{n-l}{a-u+1}.
    \end{equation}
    Similarly, it follows from $\sum_{j=1}^{m}|\mathcal{F}_j|\geq \sum_{j=1}^{m}|\hat{\mathcal{F}_j}|$ that
    \begin{equation}{\label{equ-2}} 
    \sum_{a\in R_{\gamma}}|(\overline{\mathcal{G}_\gamma})_u|\binom{n-l}{a-u}\geq\sum_{\alpha\in[m]\setminus \{\gamma\}}\sum_{b\in R_{\alpha}}|(\overline{\mathcal{G}_{\alpha}})_{l+1-u}|\binom{n-l}{b-(l+1-u)+1}.
    \end{equation}

Multiplying (\ref{equ-1}) and (\ref{equ-2}), we obtain
\begin{align}\label{equ-5}
    &\quad\sum_{\alpha\in[m]\setminus \{\gamma\}}\sum_{b\in R_{\alpha}}\sum_{a\in R_{\gamma}}|(\overline{\mathcal{G}_\gamma})_u|\cdot|(\overline{\mathcal{G}_{\alpha}})_{l+1-u}|\binom{n-l}{b-(l+1-u)}\binom{n-l}{a-u} \notag\\
    &\geq\sum_{\alpha\in[m]\setminus \{\gamma\}}\sum_{b\in R_{\alpha}}\sum_{a\in R_{\gamma}}|(\overline{\mathcal{G}_\gamma})_u|\cdot|(\overline{\mathcal{G}_{\alpha}})_{l+1-u}|\binom{n-l}{b-(l+1-u)+1}\binom{n-l}{a-u+1}.
\end{align}

Using the log-concave property of the binomial coefficients, if $n> k_1+k_2\geq a+b$, we have $\binom{n-l}{b-(l+1-u)}\binom{n-l}{a-u}<\binom{n-l}{b-(l+1-u)+1}\binom{n-l}{a-u+1}$ 
; so we have reached a contradiction. 

Note that if $n=k_1+k_2$, equality in $(\ref{equ-5})$ implies that either $R_{\gamma}=\{k_1\}$ and $R_{\alpha}=\{k_2\}$ for all $\alpha\in[m]\setminus\{\gamma\}$, or $R_{\gamma}=R_{\alpha}=\{k\}$ for all $\alpha\in[m]\setminus\{\gamma\}$. Combining (\ref{equ-1}) and (\ref{equ-2}) with Lemma \ref{keylemma}, we deduce that $m=2$. Thus, $\mathcal{F}_{\gamma}=\binom{[n]}{k_1}\setminus{\overline{\mathcal{F}_{\alpha}}}$ and $\mathcal{F}_{\alpha}\subseteq\binom{[n]}{k_2}$ with $0<|\mathcal{F}_{\alpha}|<\binom{n}{k_2}$.

\textbf{Case 2.} $u=1$ or $l$. We will show that $u=l$ cannot occur. On the contrary, assume that $u=l$(i.e., $[l]\in\mathcal{G}_{\gamma}$). Since $\{1\}\in \mathcal{G}_{\gamma}$ and $\{1\}\subsetneq [l]$, this contradicts the minimality of the elements in $\mathcal{G}_{\gamma}$. Hence, we must have $u=1$(i.e., $\{l\}\in\mathcal{G}_{\gamma}$).

Using the fact that $\mathcal{G}_{\gamma}$ contains two special generating sets, $\{1\}$ and $\{l\}$, we will deduce the value of $|\mathcal{F}_{\alpha}|$ for all $\alpha\in[m]\setminus\{\gamma\}$.
Take $E\in \mathcal{G}_{\alpha}$ for any $\alpha\in [m]\setminus\{\gamma\}$. By Lemma \ref{two-lemma}(i), we have $\{1,l\}\subseteq E$. If there exists an $x\in[l]$ such that $x\notin E$, then applying the shifting operator $s_{x,l}$ and using Lemma \ref{two-lemma}(ii), we can find $E_0\subseteq s_{x,l}(E)$ for some $E_0\in \mathcal{G}_{\alpha}$, which contradicts $l\in E_0$. 
Thus, we conclude that $\mathcal{G}_{\alpha}=\{[l]\}$ and $|\mathcal{F}_{\alpha}|=\sum_{b\in R_{\alpha}}\binom{n-l}{b-l}$ for all $\alpha\in[m]\setminus\{\gamma\}$. Since all $\mathcal{F}_{\alpha}$ for $\alpha\in[m]\setminus\{\gamma\}$ are monotone and $\mathcal{G}_{\alpha}$ contains minimal elements, we have $l\leq k_{min}^{\gamma}$. Therefore, $\mathcal{F}_{\alpha}$ is $\mathcal{M}_{2}(n,R_{\alpha},[l])$ for all $\alpha\in[m]\setminus\{\gamma\}$ with $1\leq l\leq k_{min}^{\gamma}$.

Next, we will determine $\mathcal{F}_{\gamma}$. Claim that the maximality of $|\mathcal{F}_{\gamma}|$ implies $\mathcal{F}_{\gamma}=\mathcal{M}_{1}(n,R_{\gamma},[l])$. 
By Lemma \ref{basiclemma}, it suffices to show that $\mathcal{G}_{\gamma}=\binom{[l]}{1}$. On the contrary, assume that there exists a generating set $E\in\mathcal{G}_{\gamma}$ containing distinct $x,y\in [n]\setminus\{1,l\}$ such that $\{x\},\{y\}\notin \mathcal{G}_{\gamma}$ as $\mathcal{G}_{\gamma}$ contains minimal elements and $\{1\},\{l\}\in \mathcal{G}_{\gamma}$. Consider the new family $\langle{(\mathcal{G}_{\gamma}\setminus{E})\bigcup\{x\}\bigcup\{y\}}\rangle_{R_{\gamma}}$. Since the intersection of any generating set of $\langle{(\mathcal{G}_{\gamma}\setminus{E})\bigcup\{x\}\bigcup\{y\}}\rangle_{R_{\gamma}}$ and $[l]$ is non-empty, applying Lemma \ref{two-lemma}(i) and noting that $|\langle{(\mathcal{G}_{\gamma}\setminus{E})\bigcup\{x\}\bigcup\{y\}}\rangle_{R_{\gamma}}|>|{\mathcal{F}_{\gamma}}|$, we obtain a contradiction. 

Let $F_{\gamma}(l)=|\mathcal{M}_{1}(n,R_{\gamma},[l])|+\sum_{\alpha\in [m]\setminus\{\gamma\}}|\mathcal{M}_{2}(n,R_j,[l])|$ where $1\leq l\leq k_{min}^{\gamma}$. Next we will compute the maximum of $F_{\gamma}(l)$.

\begin{lemma}{\label{maximal-f}}
     For fixed $\gamma\in[m]$, if $n>k_1+k_2$, then $F_{\gamma}(l)\leq \text{max}\{F_{\gamma}(1),F_{\gamma}(k_{min}^{\gamma})\}$. Further, equality holds if and only if, one of the following holds: 
     \begin{itemize}
         \item[(i)] if $F_{\gamma}(1)\geq F_{\gamma}(k_{min}^{\gamma})$, then $\mathcal{F}_j=\mathcal{S}(n,R_j)$ for all $j\in [m]$; 
         \item[(ii)] if $F_{\gamma}(1)\leq F_{\gamma}(k_{min}^{\gamma})$, then $\mathcal{F}_{\gamma}=\mathcal{M}_{1}(n,R_{\gamma},[k_{min}^{\gamma}])$ and $\mathcal{F}_{\alpha}=\mathcal{M}_{2}(n,R_{\alpha},[k_{min}^{\gamma}])$ for all $\alpha\in[m]\setminus\{\gamma\}$.
     \end{itemize}
\end{lemma}
\begin{proof}[Proof of Lemma \ref{maximal-f}]
    Let $l\in [1,k_{min}^{\gamma}]$ such that $F_{\gamma}(l)$ is maximum. If $l=1$ or $k_{min}^{\gamma}$, we are done. On the contrary, we assume that $1<l<k_{min}^{\gamma}$, then $F_{\gamma}(l)-F_{\gamma}(l+1)\geq0$ and $F_{\gamma}(l)-F_{\gamma}(l-1)\geq 0$.
Those imply that
\begin{equation}\label{equ-3}
\sum_{\alpha\in [m]\setminus \{\gamma\}}\sum_{b\in R_{\alpha}}\binom{n-l-1}{b-l}\geq \sum_{a\in R_{\gamma}}\binom{n-l-1}{a-1},
\end{equation}
and 
\begin{equation}\label{equ-4}
\sum_{a\in R_{\gamma}}\binom{n-l}{a-1}\geq\sum_{\alpha\in[m]\setminus \{\gamma\}}\sum_{b\in R_{\alpha}}\binom{n-l}{b-l+1}.
\end{equation}

Multiplying (\ref{equ-3}) and (\ref{equ-4}), we obtain
\begin{equation}\label{equ-6}
    \sum_{\substack{\alpha\in [m]\setminus \{\gamma\},\\b\in R_{\alpha},\\a\in R_{\gamma}}}\binom{n-l-1}{b-l}\binom{n-l}{a-1}\geq \sum_{\substack{\alpha\in [m]\setminus \{\gamma\},\\b\in R_{\alpha},\\a\in R_{\gamma}}}\binom{n-l-1}{a-1}\binom{n-l}{b-l+1}.
\end{equation}
However, if $n>k_1+k_2\geq a+b$, we have $\binom{n-l-1}{b-l}\binom{n-l}{a-1}<\binom{n-l-1}{a-1}\binom{n-l}{b-l+1}$; so we have reached a contradiction. The second part is trivial and we complete the proof. 
\end{proof}

Note that if $n=k_1+k_2$, equality in $(\ref{equ-6})$ implies that either $R_{\gamma}=\{k_1\}$ and $R_{\alpha}=\{k_2\}$ for all $\alpha\in[m]\setminus\{\gamma\}$, or $R_{\gamma}=R_{\alpha}=\{k\}$ for all $\alpha\in[m]\setminus\{\gamma\}$. If $m=2$, we are done. If $m>2$, inequality $(\ref{equ-4})$ doesn't hold. Thus, $\mathcal{F}_j=\mathcal{S}(n,R_j)$ for all $j\in[m]$. In particular, as we assume that $\mathcal{F}_j$ is L-initial, we conclude that $\mathcal{F}_j=\mathcal{F}$ for all $j\in[m]$ if $R_{j}=\{k\}$ for all $j\in[m]$, where $\mathcal{F}$ is an intersecting family with $|\mathcal{F}|=\binom{n-1}{k-1}$.

In conclusion, combining $\gamma\in[m]$ with Lemma \ref{maximal-f} yields the following result.
$$
\sum_{j=1}^{m}|\mathcal{F}_j|\leq max\left\{\sum_{j=1}^{m}|\mathcal{S}(n,R_j)|
,|\mathcal{M}_{1}(n,R_{\gamma},[k_{min}^{\gamma}])|+\sum_{\alpha\in[m]\setminus\{\gamma\}}|\mathcal{M}_{2}(n,R_{\alpha},[k_{min}^{\gamma}])|:\gamma\in[m]
\right\}
$$
and equality holds if and only if, one of the following holds:
\begin{itemize}
    \item[(i)] if $\sum_{j=1}^{m}|\mathcal{S}(n,R_j)|\geq  |\mathcal{M}_{1}(n,R_{\gamma},[k_{min}^{\gamma}])|+\sum_{\alpha\in[m]\setminus\{\gamma\}}|\mathcal{M}_{2}(n,R_{\alpha},[k_{min}^{\gamma}])|$ for all $\gamma\in[m]$, then $\mathcal{F}_j=\mathcal{S}(n,R_j)$ for all $j\in [m]$; 
    \item[(ii)] if $\sum_{j=1}^{m}|\mathcal{S}(n,R_j)|\leq  |\mathcal{M}_{1}(n,R_{\gamma},[k_{min}^{\gamma}])|+\sum_{\alpha\in[m]\setminus\{\gamma\}}|\mathcal{M}_{2}(n,R_{\alpha},[k_{min}^{\gamma}])|$ for some $\gamma\in[m]$, then $\mathcal{F}_{\gamma}=\mathcal{M}_{1}(n,R_{\gamma},[k_{min}^{\gamma}])$ and $\mathcal{F}_{\alpha}=\mathcal{M}_{2}(n,R_{\alpha},[k_{min}^{\gamma}])$ for all $\alpha\in[m]\setminus\{\gamma\}$;
    \item[(iii)] if $n=k_1+k_2$, $m=2$, $R_1=\{k_1\}$ and $R_2=\{k_2\}$, then $\mathcal{F}_1=\binom{[n]}{k_1}\setminus{\overline{\mathcal{F}_2}}$ and $\mathcal{F}_2\subseteq\binom{[n]}{k_2}$ with $0<|\mathcal{F}_2|<\binom{n}{k_2}$;
    \item[(iv)] if $n=k_1+k_2$, $m\geq 3$ and $R_i=\{k\}$ for all $i\in[m]$, then $\mathcal{F}_j=\mathcal{F}$ for all $j\in[m]$ where $\mathcal{F}$ is an
    intersecting family with $|\mathcal{F}|=\binom{n-1}{k-1}$.
\end{itemize}
The proof of Theorem \ref{maintheorem} is now complete.
\end{proof}

\section{Conclusion and Future Work}
The generating set method originated from \cite{ahlswede1997complete} and used in this paper is a powerful method to obtain upper bounds on the sum of the sizes of arbitrarily many non-empty cross-intersecting families. Furthermore, in our proof of Theorem \ref{maintheorem}, we relied heavily on Theorem \ref{keyleftlemma}. Unfortunately, there are no analogous results for cross-$t$-intersecting families when $t \geq 2$. In fact, Shi, Frankl and Qian \cite{shi2022non} posed the following problem for non-uniform cross-$t$-intersecting families:

\begin{problem}
Let $m \geq 2$ and $n, k_1, k_2, \ldots, k_m$ be positive integers. Consider non-empty cross-$t$-intersecting families $\mathcal{F}_1 \subseteq \binom{[n]}{k_1}, \mathcal{F}_2 \subseteq \binom{[n]}{k_2}, \ldots, \mathcal{F}_m \subseteq \binom{[n]}{k_m}$ such that $k_1 \geq k_2 \geq \cdots \geq k_m$ and $n \geq k_1 + k_2 - t + 1$. Is it true that  
\begin{equation*}
    \sum_{j=1}^{m} |\mathcal{F}_j| \leq \max \left\{ \sum_{j=1}^{m} \binom{n - t}{k_j - t}, \binom{n}{k_1} - \sum_{i=0}^{t-1} \binom{k_m}{i} \binom{n - k_m}{k_1 - i} + \sum_{i=2}^{m} \binom{n - k_m}{k_i - k_m} \right\}?
\end{equation*}  
\end{problem}

Gupta, Mogge, Piga and Sch\"{u}lke \cite{gupta2023r} gave a positive answer under the condition that $n \geq 2k_1 + k_{m-1} - t$ and $|\bigcap_{i=1}^{m} F_i| \geq t$ for all $F_i \in \mathcal{F}_i$ and $i \in [m]$. Similarly, we pose the following problem for generalized non-uniform families.

\begin{problem}
Let $m\geq 2$, $n$ be positive integers and $R_i=\{k_{i,1} >k_{i,2} >\cdots> k_{i,w_i}\}$ be subsets of $[n]$ for all $i=1,2,\ldots,m$. Let $\mathcal{F}_1\subseteq \binom{[n]}{R_1},\mathcal{F}_2\subseteq \binom{[n]}{R_2},\ldots,\mathcal{F}_m\subseteq \binom{[n]}{R_m}$ be non-empty cross-$t$-intersecting families such that $k_1$ (respectively, $k_2$) is the largest(respectively, second largest) value in $\{k_{1,1},k_{2,1},\ldots,k_{m,1}\}$ and $k_{min}^{\gamma}=min\{x\in R_j:j\in [m]\setminus\{\gamma\}\}$, $k_3=\min\{x\in R_j: 1\leq j\leq m\}$ and $n\geq k_1+k_2-t+1$. Is it true that
\begin{eqnarray*}
\sum_{j=1}^{m}|\mathcal{F}_j|&\leq& max \left\{ \sum_{j=1}^{m}\sum_{a\in R_j}\sum_{r\leq i\leq 2r }\binom{t+2r}{t+i}\binom{n-t-2r}{a-t-i}: 0\leq r\leq k_3-t, \right.\\&& \left.\sum_{a\in R_{\gamma}}\sum_{c=t}^{a}\binom{k_{min}^{\gamma}}{c}\binom{n-k_{min}^{\gamma}}{a-c}+\sum_{j\in [m]\setminus \{\gamma\}}\sum_{b\in R_j}\binom{n-k_{min}^{\gamma}}{b-k_{min}^{\gamma}}:\gamma\in[m]\right\}?   
\end{eqnarray*}
\end{problem}

For $m = 2$, Li, Liu, Song and Yao \cite{li2023maximum} answered the above question in the affirmative. For $m \geq 3$, the problem remains open in general.


\end{document}